\documentclass[preprint,12pt]{imsart}
\usepackage{amsthm,amsmath,natbib}
\usepackage{amsfonts}

\startlocaldefs
\renewcommand{\/}[1]{\mathbf{\boldsymbol{#1}}}
\numberwithin{equation}{section}

\newcommand{\diag}{\text{diag}}
\newtheorem{Theorem}{Theorem}
\endlocaldefs

\begin{document}

\begin{frontmatter}
\title{Shrinkage estimation with a matrix loss function}
\runtitle{Shrinkage estimation}

\begin{aug}
\author{\fnms{Reman} \snm{Abu-Shanab}\ead[label=e1]{raboshanab@sci.uob.bh}},
\author{\fnms{John T. } \snm{Kent} \corref{} 
\ead[label=e2]{j.t.kent@leeds.ac.uk}
\ead[label=u2,url]{http://www.maths.leeds.ac.uk/~john}}
\and
\author{\fnms{William E.} \snm{Strawderman}
\ead[label=e3]{straw@stat.rutgers.edu}
\ead[label=u3,url]{http://www.stat.rutgers.edu/people/faculty/straw.html}}
\affiliation{University of Bahrain, University of Leeds and Rutgers University}
\address{Department of Mathematics\\ University of Bahrain \\P.O.Box  32038\\
Kingdom of Bahrain \\ \printead{e1}}
\address{Department of Statistics\\ University of Leeds\\ 
Leeds LS2 9JT, UK \\ \printead{e2}}
\address{Department of Statistics \\561 Hill Center, Busch Campus \\
Rutgers University \\ Piscataway NJ 08854-8019 \\ \printead{e3}}
\runauthor{R. Abu-Shanab et al.}
\end{aug}

\begin{abstract}
  Consider estimating the n by p matrix of means of an n by p
  matrix of independent normally distributed observations with
  constant variance, where the performance of an estimator is
  judged using a p by p matrix quadratic error loss function. A matrix
  version of the James-Stein estimator is proposed, depending on a
  tuning constant. It is shown to dominate the usual maximum
  likelihood estimator for some choices of of the tuning constant when
  n is greater than or equal to 3.  This result also extends to other
  shrinkage estimators and settings.
\end{abstract}


\end{frontmatter}

\section{Introduction}

Shrinkage estimators are usually set in the context of \emph{vector}
data.  If $\/x (n \times 1)$ is a random vector with mean $\/\theta$,
then a shrinkage estimator of $\/\theta$ takes the form
\begin{equation}
\label{eq:vec-shrink-est}
\hat{\/\theta}_a = \/x - a \/g(\/x; u)
\end{equation}
where $a>0$ is a tuning parameter, and $\/g(\/x,u) \ (n \times 1)$ is
a ``shrinkage function'', depending on the data $\/x$, and possibly on
extra information in an auxiliary random variable (or random vector)
$u$.  Let $F(\/x,u)$ denote the joint distribution of $\/x$ and $u$,
depending on $\/\theta$.  The classic James-Stein estimator
(\citealp{Stein}; \citealp{JSE}) is a special case in the setting $\/x \sim
N_n(\/\theta, \sigma^2I_n), \ n \geq 3$.  When $\sigma^2$ is known,
the shrinkage function is given by
\begin{equation}
\label{eq:js-shr-known}
\/g(\/x) = \sigma^2 (n-2)\/x/||\/x||^2.
\end{equation}
When $\sigma^2$ is unknown, the shrinkage function is given by
\begin{equation}
\label{eq:js-shr-unknown}
\/g(\/x) =  \{u/(\nu+2)\} (n-2)\/x/||\/x||^2,
\end{equation}
where $u \sim \sigma^2 \chi^2_\nu$ is an auxiliary random variable
independent of $\/x$ which is used to estimate $\sigma^2$.

The objective in shrinkage estimation is to estimate the vector
parameter $\/\theta$, where the performance of an estimator
$\hat{\/\theta} = \hat{\/\theta}(\/x)$ is judged by the scalar loss
function
\begin{equation}
\label{eq:vector-loss}
L_{\text{scalar}}(\hat{\/\theta}, \/\theta) = 
\sum_{i=1}^n (\hat{\theta}_i - \theta_i)^2
\end{equation}
and associated risk function $R_{\text{scalar}}(\hat{\/\theta},
\/\theta) = E_F\{L_{\text{scalar}}(\hat{\/\theta}, \/\theta)\}$.

In order to guarantee that the shrinkage estimator dominates the
simple unbiased estimator $\hat{\/\theta}_0 = \/x$, the usual strategy
is to demonstrate the ``cross-product inequality''
\begin{equation}
\label{eq:cross-prod}
E_F\{(\/x-\/\theta)^T\/g\} \geq E_F(\/g^T\/g) >0,
\end{equation}
for all $\/\theta$, where $\/g = \/g(\/x,u)$ is a function of the
random vector $\/x$ (and of $u$ when present).  The last inequality
has been included to ensure that $\/g$ is nontrivial.  Throughout the
paper we assume that $\/x$ and $\/g(\/x,u)$ have finite second moments.
Then the following well-known result holds.

\begin{Theorem} Let $\/x (n \times 1)$ be a random vector and $u$ be
  an auxiliary random variable such that $E(\/x) = \/\theta$ under a
  probability model $F$ depending on $\/\theta$.  Also suppose there
  exists a shrinkage function $\/g = \/g(\/x,u)$ such that the
  cross-product inequality (\ref{eq:cross-prod}) holds.  Then the
  shrinkage estimator $\hat{\/\theta}_a$ in (\ref{eq:vec-shrink-est})
  dominates the simple estimator $\hat{\/\theta}_0 = \/x$ under the
  scalar loss function (\ref{eq:vector-loss}) provided the tuning
  parameter $a$ satisfies $0 < a < 2$.
\end{Theorem}

\begin{proof}
Write $\delta = E_F\{(\/x-\/\theta)^T\/g\}$ and $\gamma = 
E_F(\/g^T\/g)$, so $\delta \geq \gamma > 0$.  Then the risk takes the form
\begin{equation}
\label{eq:shr-proof}
\begin{split}
R_{\text{scalar}}(\hat{\/\theta}_a,
\/\theta) &= E_F \{(\/x-\/\theta-a\/g)^T(\/x-\/\theta-a\/g)\}\\
&= E_F \{(\/x-\/\theta)^T(\/x-\/\theta)\} -2a \delta + a^2 \gamma\\
&\leq E_F \{(\/x-\/\theta)^T(\/x-\/\theta)\} -2a \gamma + a^2 \gamma \\
&< E_F \{(\/x-\/\theta)^T(\/x-\/\theta)\}=
R_{\text{scalar}}(\hat{\/\theta}_0,\/\theta)
\end{split} \end{equation}
provided $0 < a < 2$.
\end{proof}

For the James-Stein estimator, Stein's Lemma \citep{Stein1} states
that the cross-product inequality for known $\sigma^2$ holds for $n
\geq 3$ and is actually an equality. That is, if $\/x \sim
N_n(\/\theta, \sigma^2 I_n),\  n \geq 3$, then
\begin{equation}
\label{eq:stein}
\sigma^2 E \left[\{\/x^T\/(\/x-\/\theta)/||\/x||^2\right\} =
(n-2)  \sigma^4 E \left\{1/|| \/x||^2\right\}= \sigma^2 A,\,
\text{say},
\end{equation}
where $A = A(\lambda^2)$ depends on $\lambda^2 = \/\theta^T\/\theta /
\sigma^2$ and $0<A<\infty$.  Stein's Lemma can be proved using
integration by parts (e.g. \citet{Efron} or \citet{Stein1}).  An
equality also holds in the analogue of (\ref{eq:stein}) for the
unknown $\sigma^2$ case since $E(u) = \nu \sigma^2, \ E(u^2) = \nu
(\nu+2) \sigma^2$ in (\ref{eq:js-shr-unknown}). Hence Theorem 1
for the James-Stein estimator, in both the known and unknown
$\sigma^2$ cases, can be strengthened to conclude that the optimal
value of the tuning constant is $a=1$, uniformly over all $\/\theta$.

The purpose of this paper is to extend James-Stein and other shrinkage
estimators to a matrix setting where the data take the form of an $n
\times p$ matrix $X$ with mean $E(X) = \Theta$.  The objective is to
estimate $\Theta$ using a $p \times p$ matrix quadratic loss
function,
\begin{equation}
\label{eq:matrix-loss}
L_{\text{matrix}}(\hat{\Theta},\Theta)=
\{\hat{\Theta}-\Theta\}^T \{\hat{\Theta}-\Theta\},
\end{equation}
and  associated risk function $R_{\text{matrix}}(\hat{\Theta},\Theta) =
E\{L_{\text{matrix}}(\hat{\Theta},\Theta)\}$.
Note that this loss function pools errors across the $n$ rows, but treats
the $p$ columns separately.  Hence we look for an estimator which 
shrinks across the $n$ rows, but does not shrink across columns.

Other authors have considered the use of shrinkage methods in a matrix
setting; see, e.g. \citet{EM,Efron,Haff,Zhe,May,Ts7,Ts9}. 
However, these papers 
use a scalar squared error loss
function and so are not directly relevant here.  There seems to be
little work focused on a matrix loss function.  

Section 2 states the main result in the matrix setting, with some
discussion given in Section 3.

\section{Matrix data}
Suppose the data take the form of an $n \times p$ matrix $X$, plus
auxiliary random variables $\/u = (u_1, \ldots, u_p)^T$, when present.
Let $\Theta = E_F(X)$ denote the $n \times p$ matrix of means, where
$F$ denotes the joint distribution of $X$ and $\/u$.  The objective is
to estimate $\Theta$ under the $p \times p$ matrix quadratic loss
function (\ref{eq:matrix-loss}).  Let $\/x_{(j)}$ denote the $j$th column
of $X$.

Suppose that for each column $j=1, \ldots, p$, there
is a shrinkage function $\/g_{(j)} = \/g_{(j)}(\/x_{(j)},u_j)$.
A natural estimator is the ``diagonal shrinkage estimator'', defined
by applying the vector shrinkage estimator separately to each column
of $X$.  That is, define  $\hat{\Theta}_a =
\hat{\Theta}_a(X)$ in terms of its columns $\hat{\/\theta}_{a,(j)}$ by
\begin{equation}
\label{DS-columns}
\hat{\/\theta}_{a,(j)} = \hat{\/\theta}_a(\/x_{(j)},u_j)
\end{equation}
using (\ref{eq:vec-shrink-est}).  Note that the shrinkage applied to
each column does not depend on the data in other columns.  We use the
term ``diagonal'' because in the setting (\ref{eq:js-shr-known}) the
estimator can also be written in matrix form using a diagonal matrix,
$$
\hat{\Theta}_a = XD, \quad D=\diag(d_j), \quad
d_j = 1-a \sigma^2 (n-2)/||\/x_{(j)}||^2, \quad j=1, \ldots, p.
$$
Given two estimators $\hat{\Theta}^{(1)}$ and $\hat{\Theta}^{(2)}$
depending on $X$, say that $\hat{\Theta}^{(1)}$ strictly dominates
$\hat{\Theta}^{(2)}$ if
$R_{\text{matrix}}(\hat{\Theta}^{(1)},\Theta) <
R_{\text{matrix}}(\hat{\Theta}^{(2)},\Theta)$ for all $\Theta$,
where ``$<$'' means that the difference between the right- and
left-hand sides is a positive-definite matrix.  The following theorem
is the main result of this paper.

\begin{Theorem}

  Let $\/X (n \times p)$ be a random matrix and $\/u = (u_1, \ldots,
  u_p)^T$ be a vector of auxiliary random variables such that $E_F(\/X)
  = \Theta$ under a probability model $F$ depending on $\Theta$,
  and the data  $\{\/x_{(j)},u_j\}$ are independent for
  different $j$.  Suppose there exist shrinkage functions $\/g_{(j)}
  = \/g_{(j)}(\/x_{(j)},u_j)$ such that the cross-product inequality
  (\ref{eq:cross-prod}) holds for each $j=1,\ldots, p$.  Then the
  shrinkage estimator $\hat{\Theta}_a$ in (\ref{DS-columns}) dominates the
  simple estimator $\hat{\Theta}_0 = \/X$ under the matrix loss
  function (\ref{eq:matrix-loss}) provided the tuning parameter $a$
  satisfies $0 < a < 2/p$.
\end{Theorem}

\begin{proof}
The proof makes use of the following inequality, where $\/\alpha$ is 
a $p \times 1$ vector and $G$ is an $n \times p$ matrix with columns
$\/g_{(j)}, \ j=1, \ldots, p$,
\begin{equation}
\label{eq:cs}
\begin{split}
\sum_{j,k=1}^p \alpha_j \alpha_k \/g_{(j)}^T \/g_{(k)}
&\leq \sum_{j,k=1}^p |\alpha_j |\; |\alpha_k| \; || \/g_{j}|| \;
|| \/g_{(k)}|| \\
&= \left\{\sum_{j=1}^p |\alpha_j|\; ||\/g_{(j)}|| \right\}^2 \\
&\leq p \sum_{j=1}^p \alpha_j^2 ||\/g_{(j)}||^2.
\end{split} \end{equation} The two inequalities follow from two
versions of the Cauchy-Schwarz inequality..

To show $\hat{\Theta}_a$ dominates $\hat{\Theta}_0 = X$ for a
particular choice of $a$, we need to show that
 $$R_{\text{matrix}}(\hat{\Theta}_a,\Theta)<
R_{\text{matrix}}(\hat{\Theta}_0,\Theta)
\text{  for all  } \Theta.
$$ 
Equivalently we need to show that 
\begin{equation}\label{lft}
\/\alpha^T R \/\alpha < n \/\alpha^T I_n \/\alpha = 
n \text{  for all  } \Theta,
\end{equation}
where $R = R_{\text{matrix}}(\hat{\Theta}_a,\Theta)$ and
$\/{\alpha}$ is an arbitrary standardized  $p$-dimensional vector,
$\/\alpha^T \/\alpha = 1$.

The left-hand side of (\ref{lft}) can be written as
\begin{equation}
\label{eq:mjs-bound}
\begin{split}
 &\sum_{j,k=1}^p  \alpha_j \alpha_k E \left\{
\left(\hat{\/\theta}_{a,(j)}-\/\theta_{(j)}\right)^T
\left(\hat{\/\theta}_{a,(k)}-\/\theta_{(k)}\right)
 \right\} \\
&=\sum_{j,k=1}^p \alpha_j \alpha_k E \left[
\left\{\left(\/x_{(j)}-\/\theta_{(j)}\right)
-  a\/g_{(j)}\right\}^T 
\left\{\left(\/x_{(k)}-\/\theta_{(k)}\right)- 
a\/g_{(k)}\right\} \right] \\
&=\sum_{j=1}^p \alpha_j^2  \left[
E \left\{\left(\/x_{(j)}-\/\theta_{(j)}\right)^T
\left(\/x_{(j)}-\/\theta_{(j)}\right)\right\}-
2 a \delta_j\right] +
a^2 \sum_{j,k=1}^p \alpha_j \alpha_k E 
\left(\/g_{(j)}^T \/g_{(k)}\right)\\
&\leq \sum_{j=1}^p \alpha_j^2 \left[
E \left\{\left(\/x_{(j)}-\/\theta_{(j)}\right)^T
\left(\/x_{(j)}-\/\theta_{(j)}\right)\right\}-
2 a \delta_j + a^2 p\gamma_j \right] \\
&\leq \sum_{j=1}^p \alpha_j^2 \left[
E \left\{\left(\/x_{(j)}-\/\theta_{(j)}\right)^T
\left(\/x_{(j)}-\/\theta_{(j)}\right)\right\}-
2 a \gamma_j + a^2 p\gamma_j \right] \\
&< \/\alpha^T R_0 \/\alpha = n, 
\end{split} \end{equation} for $0 < a < 2/p$, where $\delta_j =
E_F\{(\/x_{(j)}-\/\theta_{(j)})^T\/g_{(j)}\}$ and $\gamma_j =
E_F(\/g_{(j)}^T\/g_{(j)})$, so $\delta_j \geq \gamma_j > 0$.  In going
from the second to the third line of (\ref{eq:mjs-bound}) notice that
many of the off-diagonal terms vanish because the different columns
are independent and $E(\/x_{(j)}-\/\theta_{(j)}) = \/0$.  The fourth
line follows from the third line by the Cauchy-Schwarz based
inequality (\ref{eq:cs}). The last line follows from the fifth line by
simple properties of quadratic functions.
\end{proof}

Comments
\begin{enumerate}
\item[(a)] The allowable interval for $a$ decreases with $p$.  This
  property is related to the result that for a matrix loss function,
  it is harder to dominate the maximum likelihood estimator than for a
  scalar loss function.

\item[(b)] For the James-Stein case, the $p$-dimensional result is
  less powerful than the one-dimensional result.  In one dimension
  $a=1$ is optimal; $\hat{\/\theta}_1$ dominates $\hat{\/\theta}_a$
  for all other choices of $a$.  In contrast, if $p>1$ there is no
  single choice of $a$ for $\hat{\Theta}_a$ which dominates all other
  choices.

\item[(c)] Further, at least for the James-Stein case, the interval $(0,
  2/p)$ is the best possible interval for $a$.  If $a<0$ or $a>2/p$,
  it is possible to find values of $\Theta$ such that
  $\hat{\Theta}_a$ does not dominate $\hat{\Theta}_0$.

  Here is a simple construction in the case of known variance
  $\sigma^2=1$.  Recall $x_{ij} \sim N(\theta_{ij},1)$ independently
  for $i=1, \ldots, n, \ j=1, \ldots, p$.  Let $\alpha_j = 1/\sqrt{p},
  \ j=1, \ldots, p$.  Let $\/\theta^*$ be a $n$-vector of unit size,
  $\/\theta^{*T} \/\theta^{*} = 1$, and suppose all of the columns of
  $\Theta$ are equal to the same multiple of $\/\theta^*$,
  $\/\theta_{(j)} = \kappa \/\theta^*$.  For large $\kappa$ it is
  straightforward to show that
$$
\delta_j = \gamma_j = E(\/g_{(j)}^T \/g_{(j)}) = (m^2/\kappa^2) 
+ O(1/\kappa^4)
$$
for all $j$, where $m=n-2$.  Further (\ref{eq:cs}) becomes an equality
in this setting so that the risk in (\ref{lft}) reduces to
\begin{equation}
\label{eq:counterexample}
\/\alpha^T R \/\alpha  =n - 2a(m^2/\kappa^2) +
a^2(m^2/\kappa^2)p + O(1/\kappa^4).
\end{equation} 
Ignoring the remainder term, the quadratic function of $a$ in
(\ref{eq:counterexample}) is less than $n$ for $0 < a < 2/p$ and
exceeds $n$ for $a<0$ or $a>2/p$.  Hence for any specific choice of
$a<0$ or $a>2/p$, $\/\alpha ^T R \/\alpha>n$ for sufficiently large
$\kappa$.

The same argument works for the case of unknown $\sigma^2$.

\item[(d)] In the vector case, if the shrinkage function $\/g$ is
  re-scaled to $c \/g$ for some constant $c>0$, then the cross-product
  inequality needs minor adjustment and the allowable interval for the
  tuning parameter $a$ changes from $(0,2)$ to $(0,2/c)$.  The scaling
  convention for the cross-product inequality chosen in this paper has
  been made to make the treatment of different columns as consistent
  as possible in the extension to the matrix case.

\item[(e)] \citet{EM} proposed the ``matrix'' James-Stein estimator
$$
\hat{\/\theta}^{MJS} = X\{I_p - (n-p-1)S^{-1}\}, \quad S=X^TX,
$$
and investigated its properties under the scalar loss function
(\ref{eq:vector-loss}).  However, its properties under the matrix loss
function (\ref{eq:matrix-loss}) are unknown.

\end{enumerate}

\section{Discussion}
For the classic vector James-Stein estimator there are several
ingredients in the formulation of the problem and the estimator such
as the following: (a) normality of the data, (b) uncorrelated
components, (c) the specific choice (\ref{eq:js-shr-known}) for the
shrinkage function $\/g$, and (d) the assumption that the range of
possible values for $\/\theta$ spans all of $\mathbb{R}^n$.

Each of these ingredients can be relaxed, either individually or in
combination.  Here are some examples.
\begin{enumerate}
\item[(a)] relax normality to (i) more general spherical distributions \\
  \citep{brandwein-strawderman91, cellier-fourdrinier-95}
  or (ii) independent components \citep{shinozaki84};

\item[(b)] allow correlated normal or more general elliptic
  distributions \\ \citep{fourdrinier-strawderman-wells06};

\item[(c)] use other shrinkage estimators such as (i) subspace
  shrinkage, or more generally (ii) Bayes or generalized Bayes
  estimators based on superharmonic prior distributions
  \citep{Stein1};

\item[(d)] relax the range of possible values for
  $\/\theta$ from all of $\mathbb{R}^n$ to a specified cone 
  \citep{fourdrinier-strawderman-wells06}.
\end{enumerate}

In each case the improved performance of the shrinkage estimator is
justified by a version of the cross-product inequality.  Hence in each
case there is an immediate extension  to the matrix case.

Another direction in which the paper might be extended is to allow
dependence between the columns.  At least in the normal case with a
known $p \times p$ covariance matrix $\Sigma$, it is possible to adapt
the results of this paper.

Thus let $X (n \times p)$ follow an $np$-dimensional normal
distribution with mean $E(X) = \Theta$, with independent rows and with
common covariance matrix $\Sigma$ within each row.  Let $A$ be a matrix
square root of $\Sigma^{-1}$, so that $AA^T = \Sigma^{-1}$.  Then
$Y=XA$ has independent columns.  Hence the methodology of Section 2
can be applied to $Y$ to yield an estimator $\hat{\Phi}_a$ of $\Phi =
\Theta A$.  Back-transforming yields an estimator $\hat{\Theta}_a =
\hat{\Phi}_a A^{-1}$ which dominates $\hat{\Theta}_0$ in the matrix
sense (\ref{eq:matrix-loss}), provided $0<a<2$.

It is not clear to what extent these results carry over when $\Sigma$
needs to be estimated.  Further, note that $A$ is only defined up to a
multiplication on the left by a $p \times p$ orthogonal matrix.  Thus the
methodology of Section 2 defines a whole family of estimators, each
with the same statistical properties.  It is not clear whether it
might be possible to combine them in some way to yield a superior
estimator.
\bibliographystyle{imsart-nameyear}
\bibliography{bibliography6}  
\end{document}